\documentclass[12pt, a4paper, oneside]{amsart}
\usepackage[a4paper, margin = 2cm]{geometry}
\usepackage[english]{babel}
\usepackage[T2A]{fontenc}
\usepackage[utf8]{inputenc} 
\usepackage{amsfonts}
\usepackage{amssymb, amsthm, amscd}
\usepackage{amsmath}
\usepackage{mathtools}
\usepackage{needspace}
\usepackage{etoolbox}
\usepackage{lipsum}
\usepackage{comment}
\usepackage{cmap}
\usepackage[pdftex]{graphicx}
\usepackage[unicode]{hyperref}
\usepackage[matrix,arrow,curve]{xy}
\usepackage[usenames,dvipsnames]{xcolor}
\usepackage{colortbl}
\usepackage{textcomp}
\usepackage{cite}
\usepackage{euscript}

\DeclareFontFamily{U}{mathx}{}
\DeclareFontShape{U}{mathx}{m}{n}{<-> mathx10}{}
\DeclareSymbolFont{mathx}{U}{mathx}{m}{n}
\DeclareMathAccent{\widehat}{0}{mathx}{"70}
\DeclareMathAccent{\widecheck}{0}{mathx}{"71}

\makeatletter
\def\@settitle{\begin{center}
		\baselineskip14\p@\relax
		\bfseries
		\LARGE
		\@title
	\end{center}
}
\makeatother

\newtheorem{theorem}{Theorem}[section]

\newtheorem{lemma}[theorem]{Lemma}
\newtheorem{corollary}[theorem]{Corollary}
\newtheorem{remark}[theorem]{Remark}
\newtheorem{definition}[theorem]{Definition}
\newtheorem{mainth}{Theorem}
\newtheorem{maincor}{Corollary}

\theoremstyle{definition}
\newtheorem{example}{Example}

\renewcommand{\geq}{\geqslant}
\renewcommand{\leq}{\leqslant}
\newcommand{\R}{\mathbb{R}}
\newcommand{\N}{\mathbb{N}}
\newcommand{\Z}{\mathbb{Z}}

\newcommand{\const}{\mathrm{const}}
\newcommand{\CC}{\mathbb{C}}
\renewcommand{\C}{\mathbb{C}}

\DeclareMathOperator{\eps}{\varepsilon}

\DeclareMathOperator{\diam}{diam}

\DeclareMathOperator*{\argmin}{argmin}
\newcommand{\conv}{\text{conv}}
\newcommand{\spann}{\text{span}}
\newcommand{\clspan}{\overline{\operatorname{span}}}

\newcommand{\spec}{\text{spec}}
\newcommand{\pw}{\mathcal{P}W_\pi}

\DeclarePairedDelimiter\lr{(}{)}
\DeclarePairedDelimiter\lrsc{<}{>}

\newcommand{\eqdef}{\stackrel{\mathrm{def}}{=}}

\newcommand{\wb}{\{w r_n \}_{n \in \Z \setminus A}}

\newenvironment{enumerate*}%
  {\begin{enumerate}%
    \setlength{\itemsep}{1pt}%
    \setlength{\parskip}{1pt}}%
  {\end{enumerate}}

\makeatletter
\newcommand*{\@old@slash}{}\let\@old@slash\slash
\def\slash{\relax\ifmmode\delimiter"502F30E\mathopen{}\else\@old@slash\fi}
\makeatother

\title{Exactness of weighted exponential systems with a defect}
\author{Ivan Arefev and Andrei V. Semenov}

\thanks{The work is supported by Ministry of Science and Higher Education of the Russian Federation under agreement No.
075-15-2025-013. The second author is winner of the “Leader” competition conducted by the Foundation for the Advancement of Theoretical Physics and Mathematics “BASIS” and would like to thank its sponsors and jury.}


\address{Andrei V. Semenov
Department of Mathematics and Computer Science, Saint Petersburg State University\\
14 line of Vasilievsky island 29, Saint Petersburg, Russia 199178}

\email{asemenov.spb.56@gmail.com}

\address{Ivan Arefev
Probability Techniques in Analysis Laboratory, Saint Petersburg State University\\
14 line of Vasilievsky island 29, Saint Petersburg, Russia 199178}

\email{ivanarefiev21@gmail.com}

\begin{document}
\maketitle
\begin{abstract}
Let $w \colon (0,1) \to \mathbb{R}_{+}$ be a weight. We prove that for an arbitrary Schauder basis $\{e^{i \lambda_n t}\}_{n \in \mathbb{Z}}$ in $L^2(0,1)$ and an arbitrary lacunary defect set $A \subset \mathbb{Z}$ the system $\{w(t) r_n(t)\}_{n \in \mathbb{Z} \setminus A}$ is always complete and never minimal for any weight $w$ satisfying a natural decay condition, which is sharp on the exponential scale.\par 
  
Moreover, we establish a simple combinatorial criterion for systems of the form $ \{ w(t) e^{2\pi i n t} \}_{n \in \mathbb{Z} \setminus A} $ to be complete and minimal in $L^2(0,1)$ for an arbitrary weight $w$ and a finite defect set $A \subset \mathbb{Z}$. 

\end{abstract}

\section{Introduction}
Let $\{e^{i \lambda_n t}\}_{n \in \mathbb{Z}}$ be a complete and minimal (or {\it exact}) system in $L^2(0,1)$ and let $\Lambda = \{\lambda_n \mid n \in \Z\} \subset \R$ be its spectrum. Then there exists a unique {\it biorthogonal system} $\{r_n(t)\}_{n \in \mathbb{Z}}$ defined by the rule $\langle e^{i \lambda_n t}, r_m(t) \rangle = \delta_{m,n}$. A function $w \colon (0, 1) \to \R$ is called a {\it weight} if it is measurable and positive almost everywhere. Consider the weighted system $E = \{r_n(t) w(t)\}_{n \in \Z \setminus A}$ for some $A \subset \Z$. The set $A$ is called the {\it defect} of the system $E$.  The natural question arises: how fast should the weight $w$ decay and how large can the defect $A$ be for the system $E$ to be complete and minimal in $L^2(0,1)$?

In 2012, Heil and Yoon studied this problem for the orthonormal basis $\{e^{2 \pi i n t}\}_{n \in \Z}$ and a polynomial weight $w(t) = t^N$ for $N \in \N$. They showed that if the system $E = \{ t^N e^{2 \pi i n t} \}_{n \in \Z \setminus A} $ is complete in $L^2(0,1)$, then $|A| \leq N$ and if the system $E$ is minimal, then $|A| \geq N$. They also proved the converse needed for exactness: the system $E$ is complete and minimal if and only if $|A| = N$. Furthermore, their Theorem 8(c) provides a criterion for exactness of systems of the form $\{ w(t) e^{2\pi i n t} \}_{n \in \mathbb{Z} \setminus A}$ for an arbitrary $w \in L^2(\mathbb{T})$.\par
Later in 2018 Shukurov in \cite{Shukurov} considered the case of an arbitrary weight $w$ and a singleton defect set with $|A| = 1$. He showed that for an arbitrary $k_0 \in \Z$ the system $\{ w(t) e^{i n t} \}_{n \ne k_0}$ is complete and minimal in $L^p(-\pi, \pi)$ with $1 < p < +\infty$ if and only if $w \in L^p(-\pi, \pi)$ and $\frac{1}{w} \notin L^q(-\pi, \pi)$ for $1/p+1/q=1$, and one of the following two conditions holds:
\begin{enumerate}
    \item There is a unique $t_0 \in [-\pi, \pi]$ such that $\frac{t - t_0}{w(t)} \in L^q(-\pi, \pi)$,
    \item $ \frac{(t - \pi)(t + \pi)}{w(t)} \in L^q(-\pi, \pi) $.
\end{enumerate}

Finally, in 2025, Zikkos in \cite{Zikkos} generalized the ideas of Yoon and Heil: he considered the case of an arbitrary Schauder basis $\{e^{i \lambda_n t}\}_{n \in \Z}$ and a power weight $w(t) = t^{\alpha}$ for $\alpha > 0$. For a defect set $A\subset \Z$ such that $1\leq |A| = M < \infty$ he showed that $\alpha \geq M - \frac{1}{2}$ implies the completeness of the system $E = \{ t^\alpha r_n(t)\}_{n \in \Z \setminus A}$, while $\alpha < M + \frac{1}{2}$ implies the minimality of $E$. Moreover, the system $E$ is complete and minimal if and only if $\alpha \in [M - \frac{1}{2}, M + \frac{1}{2})$.

Therefore, the case of a finite defect set and weights with polynomial decay is covered. It is not hard to see that if $w$ has super-polynomial decay at 0, then the system $\wb$ is always complete and never minimal for any finite defect set $A$ (see Lemma \ref{th:finite}). For weights with super-polynomial decay, the corresponding question for infinite defect sets remains open. For general weights and finite defects in the standard orthonormal basis, Heil and Yoon gave an exactness criterion in terms of $L^2$-zeros and polynomial interpolation. 
\begin{theorem} [Heil, Yoon, \cite{Heil}] 
    Given $w \in L^2(\mathbb{T})\setminus\{0\}$ and a nonempty finite set $A \subset \Z$, the system $\{ w(t) e^{2\pi i n t} \}_{n \in \Z \setminus A}$ is complete and minimal in $L^2(\mathbb{T})$ if and only if there exist pairwise distinct points $\alpha_1,\ldots,\alpha_k\in\mathbb T$, where $k\ge1$, and positive integers $n_1,\ldots,n_k$ such that
    \begin{enumerate}
        \item $w$ has an $L^2$-zero at $\alpha_i$ of order $n_i$ for each $1 \leq i \leq k$,
        \item $(1/w)\chi_I \in L^2(\mathbb{T})$ for any closed interval $I \subset \mathbb{T}\setminus\{ \alpha_i \}_{i = 1}^k$,
        \item $n_1 + \dots + n_k \geq |A|$, 
        \item For each $m \in \Z \setminus A$ there exists a unique trigonometric polynomial $p_m$ of the form 
$$p_m(t) = e^{2\pi i m t} + \sum\limits_{k \in A}c_{m, k} e^{2\pi i k t}$$
        such that $p_m^{(j)}(\alpha_i) = 0$ for $1 \leq i \leq k$ and $0 \leq j \leq n_i - 1$.

    \end{enumerate}
\end{theorem}
However, their criterion requires the existence and uniqueness of an interpolating polynomial for every retained frequency, leaving the admissible combinations of zero locations, multiplicities, and removed frequencies described by an infinite family of interpolation conditions.

 We cover the case of an arbitrary complete and minimal system $\{e^{i \lambda_n t}\}_{n \in \Z}$ and an infinite defect set $A \subset \Z$ such that the set $\Lambda(A) = \{\lambda_n \mid n \in A\}$ is lacunary. In the case of a finite defect set $A \subset \Z$ we find an explicit combinatorial form of the exactness criterion. In particular, we show that completeness and minimality force total multiplicity of the $L^2$-zeros of $w$ on $\mathbb T=\mathbb R/\mathbb Z$ to equal $|A|$. We further express the remaining interpolation condition as the nonvanishing of an explicitly specified Schur polynomial evaluated at the corresponding points of the unit circle. The examples in Section 8 show how the Schur-polynomial criterion yields explicit geometric characterizations of admissible zero configurations and clarify the difference between the uniform diameter bound and optimal thresholds for particular defect sets.

We also want to note that the study of exactness for weighted exponential systems is motivated by a natural interest in the boundaries of basis properties. The fundamental question of whether there exist conditional Schauder bases in $L^2(-\pi, \pi)$ among natural trigonometric systems was posed by N. K. Bari in \cite{Bari}. A classical answer was given by K. I. Babenko in \cite{Babenko}, where he proved that the weighted system $\{|t|^\alpha e^{i n t}\}_{n \in \Z}$ forms precisely such a basis for $-1/2 < \alpha < 1/2$ ($\alpha \neq 0$). This result was further generalized by V.F. Gaposhkin \cite{Gaposhkin}. It was established by Hunt, Muckenhoupt, and Wheeden (see \cite{Hunt}), as well as Hunt and Young (see \cite{Young}), that a general weighted exponential system forms a Schauder basis if and only if the weight squared satisfies the Muckenhoupt $A_2$ condition. However, if the weight violates the $A_2$ condition or if a defect set of frequencies is removed from the system, the basis property is inevitably lost. This naturally leads to the question of whether the system at least preserves exactness under such strong perturbations.

\section{Main results}

All limits and estimates involving the weight $w$ are understood up to a set of measure zero. Hence, we will identify $\sup$ and $\operatorname{ess\,sup}$, and $\lim$ and $\operatorname{ess\,lim}$ for the sake of presentation.

\subsection{The case of super-polynomial decay}
Let $w \in L^{\infty}(0, 1)$ be a weight with super-polynomial decay at zero. It appears that the case of a finite defect set $A \subset \Z$ is trivial.

\begin{lemma}\label{th:finite}
  Let $\{e^{i \lambda_n t}\}_{n \in \mathbb{Z}}$ be a Schauder basis in $L^2(0,1)$, and let $\{r_n(t)\}_{n \in \mathbb{Z}}$ be its dual basis. Then the system $\{ w(t) r_n(t) \}_{n \in \Z \setminus A}$ is always complete and never minimal in $L^2(0,1)$ for any finite set $A$ and any weight $w \in L^{\infty}(0,1)$ such that 
$$w(t) = o(t^n) \text{ for any } n \in \N \text{ as } t \to 0.$$
\end{lemma}

\begin{corollary}
        Under the same assumptions, the system $\{ w(t) e^{2\pi int} \}_{n \in \Z \setminus A}$ is always complete and never minimal in $L^2(0,1)$ for any finite set $A$.
\end{corollary}

Now consider the case of an infinite defect set $A = \{m_k \}_{k \in \Z} \subset \mathbb{Z}$. We define $\Lambda(A) := \{\lambda_k \mid k \in A\}$.
\begin{definition}
An unbounded set $\Lambda \subset\mathbb R$ is called two-sided lacunary (or simply lacunary) if it admits a unique monotone enumeration
\[
 \ldots <\lambda_{-2}<\lambda_{-1} < 0\leq \lambda_1<\lambda_2<\ldots
\]
for which
\[
 1<m^+=\liminf_{j\to + \infty}\frac{\lambda_{j + 1}}{\lambda_j}<+\infty
 \qquad \text{and} \qquad
 1<m^-=\liminf_{j\to - \infty}\frac{|\lambda_{j - 1}|}{|\lambda_j|}<+\infty.
\]
The constants $m^+$ and $m^-$ are called lacunarity constants of $\Lambda$.
\end{definition}
For example, the set $\{2^n \mid n \in \N\} \cup \{-2^n \mid n \in \N\}$ is lacunary with lacunarity constants $m^- = m^+ = 2$. Or main results in case of superpolynomial decay are two following theorems.

\begin{mainth}\label{th:main}
Let $\{e^{i \lambda_n t}\}_{n \in \mathbb{Z}}$ be a complete and minimal system in $L^2(0,1)$ and let $\{r_n(t)\}_{n \in \mathbb{Z}}$ be its biorthogonal system. Let $A \subset \Z$ be such that $\Lambda(A) = \{ \lambda_m \mid m \in A \}$ is lacunary with lacunarity constants $m^{\pm} > 1$. Then the system $\{w r_n\}_{n \in \mathbb{Z} \setminus A}$ is always complete and never minimal in $L^2(0,1)$ for any weight $w \in L^{\infty}(0,1)$ such that 
\begin{equation} \label{eq:weight_dec}
    \lim_{t \to 0} w(t) \cdot \exp\left(  \frac{1+\varepsilon}{2} \left({{1} \over {\ln m^{+}}} + {{1} \over {\ln m^{-}}}\right)\ln^2 t \right) = 0
\end{equation}
for some $\varepsilon>0$.
\end{mainth}

The condition on decay is essential. The following theorem shows that \eqref{eq:weight_dec} is sharp on the exponential scale.

\begin{mainth}\label{th:contr}
For every integer $m \ge 2$ there exists a set $A \subset \Z$ such that $\Lambda(A) = 2\pi A$ is lacunary with lacunarity constants $m^\pm = m \in \N_{> 1}$, and there exists a weight $w \in L^{\infty}(0,1)$ such that 
$$\lim_{t \to 0^+} w(t) \cdot \exp\left(  \frac{1}{\ln m} \ln^2 t \right) = 0,$$
but the system $\{ w(t) e^{2 \pi i n t}  \}_{n \in \Z \setminus A}$ is not complete in $L^2(0,1)$.
\end{mainth}

\subsection{The case of a weight with multiple zeros}
It is known that if $A = \varnothing$, then completeness and minimality of the system $\{ w(t) e^{2\pi i n t} \}_{n \in \Z}$ are equivalent to $w \in L^2$ and $1/w \in L^2$ (see, for example, Lemma \ref{lem:L2_weight}). From now on we assume that the defect set is non-empty.
\begin{definition} \label{def:zero_set}
Let $\mathbb T=\R/\Z$. For a finite multiset $Z=\{t_j\}_{1\leq j\leq N}\subset\mathbb T$ put $D_Z(t)=\prod_{j=1}^N\lr*{e^{2\pi i t}-e^{2\pi i t_j}}$. The multiset $Z$ is called the \emph{set of zeros of the weight $w$} if
\begin{equation}\label{eq:defzero1}
\frac{D_Z}{w}\in L^2(0,1) 
\end{equation}
and for any $t_d\in Z$ we have
\begin{equation}\label{eq:defzero2}
\frac{D_Z(t)}{w(t)\lr*{e^{2\pi i t}-e^{2\pi i t_d}}}\notin L^2(0,1).
\end{equation}
The multiplicity of a distinct zero $t_j$ will be denoted by $a_j$, and $|Z|$ is the sum of these multiplicities. If no such finite multiset exists, the zero set is said to be undefined. We also say that $Z = \varnothing$ if $1/w \in L^2(0, 1)$.
\end{definition}
If such a multiset $Z$ exists, then it is uniquely determined (see Lemma \ref{lem:proving_def}).
\begin{example}
    Note that our definition identifies the zeros at $0$ and $1$. For example, for $w = t(1 - t)$, the corresponding polynomial is $D_Z = e^{2\pi i t} - 1$, and hence $|Z| = 1$.
\end{example}
We use the convention that every partition is written in nondecreasing order.
\begin{mainth} \label{th:fin_schur}
    Consider a weight $w$ and a defect set $A = \{ k_1, \ldots, k_m \}$ with $k_1 < \ldots < k_m$. The system $\{ w(t)  e^{2\pi i n t}  \}_{n \in \Z \setminus A}$ is complete and minimal in $L^2(0,1)$ if and only if 
    \begin{enumerate}
        \item $w \in L^2(0, 1)$,
        \item $w$ has a finite zero multiset $Z=\{t_j\}_{j=1}^{m}$ so that $|Z| = |A|$ (counting multiplicities),
        \item we have 
          \[
        s_{A'} (e^{2 \pi i t_1}, e^{2\pi it_2}, \dots, e^{2\pi it_m}) \neq 0,
    \]
    where $s_{A'}$ is the Schur polynomial corresponding to the partition $A' = ( k_j - k_1 + 2 - j )_{ 1 \leq j \leq m }$.
    \end{enumerate}
\end{mainth}

Theorem~\ref{th:fin_schur} reformulates the interpolation criterion of Heil and Yoon as the nonvanishing of a single explicitly specified Schur polynomial. Once the zero multiset is known, this gives a finite
algebraic condition on the removed frequencies and the zero locations.
It also yields the following sufficient condition in terms of the diameter of the zero multiset.

\begin{definition}
    The diameter of a multiset $Z = \{ t_j \}_{1 \leq j \leq m} \subset \R$ with respect to $\R/\Z$ is defined as the normalized length of the shortest contiguous arc on the unit circle containing all the points $e^{2\pi i t_j}$. Equivalently,
    \[
    \diam Z = \min_{l_1, \dots, l_m \in \Z} \left( \max_{1 \leq i, j \leq m} |(t_i - l_i) - (t_j - l_j)| \right).
    \]
\end{definition}

\begin{maincor} \label{cor:schur}
        Let $w \in L^2(0, 1)$ be a weight with a multiset of zeros $Z = \{t_j\}_{1 \leq j \leq m}$ for some $m \geq 2$, and let $A = \{ k_1 < \ldots < k_m \}$. We use the convention $0^{-1} = +\infty$. If
        \[
        \diam Z < \lr*{ 2 \sum_{j = 1}^{\left\lfloor \frac{m}{2} \right\rfloor} (k_{m - j + 1} - k_j) -  \left\lfloor \frac{m^2}{2} \right\rfloor }^{-1},
        \]
         then the system $\{ w(t)  e^{2\pi i n t}  \}_{n \in \Z \setminus A}$ is complete and minimal.
\end{maincor}
\begin{remark}
    For $|Z| = |A| = 1$, we have $\diam Z = 0$, so the completeness and minimality of the system $\{ w(t) e^{2\pi i n t}\}_{n \in \Z \setminus A}$ follow immediately from Theorem \ref{th:fin_schur}. Thus, this trivial case is omitted. 
\end{remark}

Using this corollary, one can recover Heil and Yoon's result.

\begin{example}
    Let $|A| = m$. If $w(t) = t^m$, then $t_1 = \ldots = t_m = 0$. Hence $\diam Z = 0$ and the condition in Corollary \ref{cor:schur} holds trivially. This recovers Theorem 4 in \cite{Heil}.
\end{example}

The Corollary \ref{cor:schur} allows us to answer the questing for a different distribution of zeros (see Section 8 for details).
\begin{example}
    Let $w \in L^2(0,1)$ be a weight. If $A = \{ a + j \}_{0 \leq j \leq m - 1}$ for some $a \in \Z$, then $2 \sum_{j = 1}^{\left\lfloor \frac{m}{2} \right\rfloor} (k_{m - j + 1} - k_j) -  \left\lfloor \frac{m^2}{2}\right\rfloor = 0$ and the system $\{ w(t)  e^{2\pi i n t}  \}_{n \in \Z \setminus A}$ is complete and minimal if $|Z| = |A|$ (counting multiplicities).
\end{example}
 The uniform bound in Corollary \ref{cor:schur} provides a sufficient, but generally not necessary, geometric condition. For specific defect sets, the exact Schur-polynomial criterion in Theorem \ref{th:fin_schur} yields sharper thresholds. For instance, for the defect set $A=\{0, 1, 3\}$ the system is exact if and only if $\diam Z < 2/3$, which strictly improves upon the uniform guarantee $\diam Z < 1/2$. Further examples and detailed comparisons are deferred to Section 8.

 \subsection{Usage of Large Language Models} The main idea, all the constructions and proofs belongs solely to authors. In early steps of the research, the authors used Google Gemini 3.1 Pro as an exploratory aid in surveying the literature. During the preparation of the paper, the authors used ChatGPT 6.0 Astra and Claude Opus 5.0 to check for grammatical errors and typos and to check the correctness of the proof.

\section{Preliminaries}
In this section we list the basic definitions and facts that we use. We denote the floor function by $\lfloor x \rfloor = \max \{z \in \Z \mid z \leq x\}$, and the fractional part of $x \in \R$ by $\{x\} = x - \lfloor x \rfloor$. We normalize the Fourier transform as follows:
\[
(\mathcal Fg)(z)={{1} \over {\sqrt{2\pi}}} \int_{-\pi}^{\pi}g(t)e^{itz} dt.
\]
Let us consider the Paley--Wiener space $\pw = \mathcal{F} L^2(-\pi,\pi)$ with reproducing kernels $K_\lambda$ and denote its biorthogonal system by $\{ g_\lambda \}_{\lambda \in \Lambda}$, following the notation introduced in \cite{BBB}. Recall that the upper density of the set $\Lambda$ is given by
$$
D_+(\Lambda) = \limsup_{r\to\infty} \frac{n_r(\Lambda)}{2r},
$$
where $n_r(\Lambda) =|\{ {\lambda\in\Lambda \mid |\lambda|\leq r}\}|$.
\begin{theorem} [Theorem 1.2 in \cite{BBB}] \label{BBB:th}
Let $\Lambda \subset \mathbb C$ be such that the system $\{K_\lambda\}_{\lambda\in \Lambda}$ is complete and minimal in $\pw$ and let the partition $\Lambda = \Lambda_1 \cup \Lambda_2$ satisfy $D_+(\Lambda_1) >0$.
Then the system $\{g_\lambda\}_{\lambda\in \Lambda_1} \cup
\{K_\lambda\}_{\lambda\in \Lambda_2}$ is complete in $\pw$. 
\end{theorem}

\begin{theorem}[Ingham inequality, see Theorems 1, 2 in \cite{Ingham}] \label{th:Ing}
    Let a monotone sequence $\{\lambda_n\}_{n \in \Z} \subset \R$ be such that $\lambda_{ n + 1 } - \lambda_n  \geq \gamma > 0 $ for all $n \in \Z$. Then
    \begin{enumerate}
        \item For all $T > \frac{\pi}{\gamma}$ there exists a constant $C_1 = C_1(\gamma, T)> 0$ such that for any finite sequence $\{ a_n \}_{n \in \Z}$ we have the inequality:
    \[
    C_1 \sum\limits_{n \in \Z} |a_n|^2 \leq \left\|  \sum\limits_{n \in \Z} a_n e^{i \lambda_n t}  \right\|^2_{L^2(-T, T)}.
    \]
    \item For all $T > 0$ there exists a constant $C_2 = C_2(\gamma, T)> 0$ such that for any finite sequence $\{ a_n \}_{n \in \Z}$ we have the inequality:
    \[
    \left\|  \sum\limits_{n \in \Z} a_n e^{i \lambda_n t}  \right\|^2_{L^2(-T, T)} \leq C_2 \sum\limits_{n \in \Z} |a_n|^2.
    \]
    \end{enumerate}
\end{theorem}

\begin{definition}\label{def:omega}
Let $p>0$. For $f \in L_{loc}^p(\R)$, we define $\omega_f^{(p)}: [0, + \infty) \to [0, + \infty)$ as follows:
    \[
    \omega_f^{(p)} (t) := \inf\limits_{J} \lr*{  \frac{1}{|J|}  \int\limits_J |f(x)|^p dx }^{\frac{1}{p}},
    \]
    where $J$ is any bounded subinterval of $\R$ such that $|J| \geq t$ and $|J| > 0$.
\end{definition}

\begin{definition}
    Let $\Lambda \subset \R$ be a set without finite limit points. For $\lambda \in \Lambda$ and $t \geq 0$ define
    \[
    Y_{\lambda}(t) = |(\Lambda \setminus \{ \lambda \} - \lambda)\cap [-t, t]|.
    \]
\end{definition}

To formulate the theorems proved by Belov in \cite{Belov}, we denote by \eqref{ineq:A} and \eqref{ineq:B} the following relations:
\begin{equation}
    \int\limits_0^{+ \infty} \frac{Y_\lambda(t)}{t^2}dt < +\infty \tag{A} \label{ineq:A}
\end{equation}
and
\begin{equation}
    \liminf_{\Delta \to +\infty} \omega\lr*{ \pi (1 + \gamma) \int\limits_\Delta^{+ \infty} \frac{Y_\lambda(t)}{t^2}dt } \exp \lr*{ \int\limits_0^{\Delta} \frac{Y_\lambda(t)}{t}dt } = 0 \tag{B} \label{ineq:B},
\end{equation}
where $\omega:\R_{\geq 0} \to \R_{\geq 0}$ is non-decreasing function, $\gamma > 0, \ \Lambda \subset \R$ and $\lambda \in \Lambda$.
\begin{theorem} [Theorem 8, \cite{Belov}] \label{belov:th_8}
Consider a function $f \in L_{loc}^1(\R)$ and a set $\Lambda \subset \R$ with no finite limit points, and suppose that for some $\lambda_0 \in \Lambda$ inequality \eqref{ineq:A} holds. Also let a sequence of finite sums $f_j(x) = \sum\limits_{\lambda \in \Lambda} c_{\lambda}^{(j)} e^{i \lambda x}$ converge to $f$ in $L_{loc}^1(\R)$. If condition \eqref{ineq:B} holds with $\lambda=\lambda_0$ for $\omega = \omega^{(1)}_f$ and some $\gamma > 0$, then for all $\lambda \in \Lambda$ we have $c_{\lambda}^{(j)} \to 0$ as $j \to +\infty$.
\end{theorem}

\begin{lemma}[Lemma 21, \cite{Belov}] \label{belov:lem_21}
    Let $\Lambda \ni 0$ be lacunary with lacunarity constants $m^\pm>1$ and $\omega:\R_{\geq 0} \to \R_{\geq 0}$ is non-decreasing. If for some $\varepsilon > 0$ we have
    \[
    \liminf_{t \to 0^+} \omega \lr*{t} \exp\left(  \frac{1+\varepsilon}{2} \left({{1} \over {\ln m^{+}}} + {{1} \over {\ln m^{-}}}\right)\ln^2 t \right) = 0,
    \]
    then conditions \eqref{ineq:A} and \eqref{ineq:B} hold with $\lambda = 0$.
\end{lemma}

\section{Auxiliary Propositions}
To prove Theorem \ref{th:main} and Lemma \ref{th:finite}, we need several important results on the decay of some functions in $L^2(0,1)$.

\begin{lemma} \label{lem:main}
Let $\widetilde{f}$ denote any arbitrary extension of $f \in L^2(0, 1)$ to the real line $\R$ satisfying $\tilde{f} \in L_{loc}^1(\R)$. Suppose there exists $\varepsilon_1>0$ such that
    \[
    \lim_{t\to0^+}w(t)\exp\left(\frac{1+\varepsilon_1}{2}
    \left(\frac1{\ln m^+}+\frac1{\ln m^-}\right)\ln^2t\right)=0.
    \]
    Then for any $f \in L^2(0, 1)$ such that $f/w \in L^2(0, 1)$, there exists $\varepsilon_2>0$ for which
    \[
    \lim_{s\to0^+}\omega_{\widetilde{f}}^{(1)}(s)
    \exp\left(\frac{1+\varepsilon_2}{2}
    \left(\frac1{\ln m^+}+\frac1{\ln m^-}\right)\ln^2s\right)=0.
    \]
\end{lemma}

\begin{proof}
    Let $ \| f/w \|^2_{L^2(0, 1)} = M < \infty$. In Definition \ref{def:omega} we set $p=1$ and $J = (0, t)$. Then we have
$$
\omega^{(1)}_{\widetilde{f}}(t) \leq \frac{1}{t} \int\limits_{0}^{ t} |\widetilde{f}(x)| dx = \frac{1}{t} \int\limits_{0}^{t} |f(x)|dx \leq t^{-1/2} \lr*{\int\limits_{0}^{t} |f(x)|^2 dx}^{\frac{1}{2}} .
$$
The integral can be estimated as follows:

\[
    \int_0^{t} |f(x)|^2 dx = \int_0^{t} \frac{|f(x)|^2}{w(x)^2} w(x)^2 dx \leq \sup_{x \in (0, t)} w(x)^2 \cdot \int_0^{t} \left|\frac{f(x)}{w(x)}\right|^2 dx \leq M \cdot \sup_{x \in (0, t)} w(x)^2.
\]

Therefore, we obtain
$$
\omega^{(1)}_{\widetilde{f}} (t) \leq \sqrt{\frac{M}{t}} \sup_{x \in (0, t)} w(x).
$$
Let us denote $K = \frac{1}{\ln m^+} + \frac{1}{\ln m^-} > 0$. The condition on decay of $w$ implies that
$$w(x) \leq \exp\left(-\frac{1+\eps_1}{2} K \ln^2 \frac{1}{x}\right)$$
for sufficiently small $x > 0$. Since this majorizing function is continuous and monotonically increasing with respect to $x$, its supremum on the interval $(0, t)$ is attained at the right endpoint. Then we have
\begin{gather*}
    \omega^{(1)}_{\widetilde{f}}(t) \leq \sqrt{\frac{M}{t}} \sup_{x \in (0, t)} w(x) \leq \sqrt{\frac{M}{t}} \sup_{x \in (0, t)} \exp\left(-\frac{1+\eps_1}{2} K \ln^2 \frac{1}{x}\right) \leq \\
    \leq \sqrt{\frac{M}{t}} \exp\left(-\frac{1+\eps_1}{2} K \ln^2 \frac{1}{t}\right) = \sqrt{M} \exp\left( \frac{1}{2}\ln \frac{1}{t} -\frac{1+\eps_1}{2} K \ln^2 \frac{1}{t} \right).
    \end{gather*}
Let $\eps_2$ be such that $0 < \eps_2 < \eps_1$. We obtain 
$$
0 \leq \omega^{(1)}_{\widetilde{f}}(t) \exp\left( \frac{1+\eps_2}{2} K \ln^2 \frac{1}{t} \right) \leq \sqrt{M} \exp\left( \frac{1}{2}\ln \frac{1}{t} -\frac{1+\eps_1}{2} K \ln^2 \frac{1}{t} + \frac{1+\eps_2}{2} K \ln^2 \frac{1}{t} \right)  \leq
$$
$$
\leq \sqrt{M} \exp\left( \frac{1}{2} \ln \frac{1}{t} - \frac{\eps_1 - \eps_2}{2} K \ln^2 \frac{1}{t} \right).
$$
Set $z = \frac{1}{t} \to + \infty$ as $t \to 0^+$. Since $\eps_1 > \eps_2$, we have $\left( \frac{1}{2}\ln z - \frac{\eps_1 - \eps_2}{2} K \ln^2 z \right) \to -\infty$ as $z \to +\infty$. Thus we have
$$\lim_{z \to +\infty} \omega^{(1)}_{\widetilde{f}}(1/z) \exp\left( \frac{1+\eps_2}{2} K \ln^2 z \right) = 0.$$
\end{proof}

\begin{definition}
    Let $\{ e^{i \lambda_n t} \}_{n \in \Z}$ be a Schauder basis. Then the spectrum of the function $f=\sum_{n \in \Z} c_n e^{i \lambda_n t}$ is defined as $\spec(f) = \{ n \mid c_n \neq 0 \}$.
\end{definition}

\begin{lemma}\label{lem:general}
Let $\{e^{i\lambda_nt}\}_{n \in \Z}$ be a Schauder basis in $L^2(0, 1)$ and let $\{r_n\}_{n \in \Z}$ be its dual basis. Fix a set $A \subset \Z$ and an arbitrary weight $w \in L^\infty(0, 1)$. Then
    \begin{enumerate}
        \item $\wb$ is complete in $L^2(0, 1)$ if and only if
        \[
            \frac{f }{ w } \notin L^2(0, 1) \setminus \{0\}
        \]
        for every $f \in L^2(0, 1)$ such that $\spec (f) \subset A$.
        \item $\wb$ is minimal in $L^2(0, 1)$ if and only if for every $n \notin A$ there exists a function $f_n \in L^2(0, 1)$ with $n \in \spec (f_n) \subset A \cup \{ n \}$ and
        \[
            \frac{f_n}{ w } \in L^2(0, 1).
        \]
    \end{enumerate}
\end{lemma}

\begin{proof}
    By definition, $\wb$ is complete if and only if for any $g \in L^2(0, 1)$ the condition $\lrsc*{g, wr_n} = 0$ for all $n \notin A$ implies $g = 0$. Fix an arbitrary function $g \in L^2(0,1)$. Since $w$ is real-valued we obtain $\lrsc*{g, wr_n} = \lrsc*{gw, r_n}$. By hypothesis we have $gw \in L^2(0, 1)$. Then the condition $\lrsc*{g, wr_n} = 0$ for all $n \notin A$ is equivalent to the following equation:
        \[
    g(t)w(t) = \sum\limits_{n \in \Z} \lrsc*{gw, r_n} e^{i \lambda_n t} = \sum\limits_{n \in \Z} \lrsc*{g, wr_n} e^{i \lambda_n t} = \sum\limits_{m \in A} \lrsc*{g, wr_m} e^{i \lambda_m t},
    \]
    where the order of summation determined by the Schauder basis. Thus we obtain
$$g(t) =  \frac{\sum\limits_{ m \in A } c_{m} e^{i \lambda_{m} t} }{ w(t) } = \frac{f(t)}{w(t)} ,$$
where $c_m = \lrsc*{gw, r_m}$ and the equality is understood in the sense of $L^2(0, 1)$. Since the existence of such a $g$ is equivalent to the fact that the right-hand side is in $L^2(0,1) \setminus \{ 0\}$,  the system $\wb$ is complete if and only if for any $f = \sum\limits_{ m \in A } c_{m} e^{i \lambda_{m} t} \in L^2(0, 1)$ the following implication holds:
    \[
    \frac{f }{ w } \in L^2(0, 1)  \Longrightarrow f = 0 \text{ almost everywhere}.
    \]
The minimality of $\wb$ is equivalent to the existence of its biorthogonal system $\{g_n\}_{n \in \Z \setminus A} \subset L^2(0, 1)$. As in the previous step, this is equivalent to requiring that for any $n \notin A$ there exists $g_n \in L^2(0, 1)$ such that 
\[
\lrsc*{wg_n, r_m} = \begin{cases}
    0, & m \notin A \  \text{ and } \ m \neq n, \\
    1, & m \notin A \ \text{ and } \ m = n.
\end{cases}
\]
This is equivalent to $g_n$ having the representation:
    \[
    g_n(t) = \frac{e^{i\lambda_nt} + \sum\limits_{ m \in A } \xi_{n, m} e^{i \lambda_{m} t}}{ w(t) } = \frac{f_n(t)}{w(t)},
    \]
where the fraction on the right-hand side of the equality exists as an element of $L^2(0,1)$.
\end{proof}
This allows us to immediately prove Lemma \ref{th:finite}:
\begin{proof}
        Indeed, any linear combination of exponentials from $\{e^{i \lambda_n t}\}_{n \in \mathbb{Z}}$ has zeros of only finite order. It is obvious that for any non-trivial finitely supported sequence $\{ c_n\}_{n \in \Z}$ one has $\frac{\sum_{n \in \Z} c_n e^{i \lambda_n t} } {w(t)} \notin L^2(0, 1)$, since $w$ has superpolynomial decay at zero.
    Therefore, the system $\wb$ is complete and not minimal for any finite $A \subset \Z$ by Lemma \ref{lem:general}.
\end{proof}
    
\section{Proof of Theorem \ref{th:main}}
We split the proof of the theorem into three steps.

\subsection{Auxiliary Lemmas} First, we need a simple corollary of Theorem \ref{BBB:th}.
\begin{lemma}\label{lem:hereditary}
    Let $\{ e^{i \lambda_n t} \}_{n \in \Z}$ be a complete and minimal system in $L^2(0, 1)$ and let $\{ r_n \}_{n \in \Z}$ be its biorthogonal system. Also let $ A \subset \Z$ be such that $\Lambda(A) = \{ \lambda_n \mid n \in A \} \subset \R$ is lacunary. Then if $\langle f, r_n\rangle = 0$ for all $n \notin A$, then $f \in \clspan \{ e^{i\lambda_n t} \mid n \in A\} $.
\end{lemma}

\begin{proof}
By the Cartwright--Levinson theorem, with the normalization corresponding to $L^2(0,1)$ we have $\lim_{R\to\infty}\frac{n_R(\Lambda)}{2R}=\frac{1}{2\pi}$, see \cite[Lecture 17, Theorem 1 and  Lecture 18, Theorem 2]{Levin}. Since $\Lambda(A)$ is lacunary, $n_r(\Lambda(A))=O(\log r)$. Consequently, $D_+\lr*{\Lambda(\Z\setminus A)} = 1/(2\pi)>0$. For $x \in (-\pi, \pi)$ define
\[
(Uf)(x) = (2\pi)^{-1/2}f\lr*{ \frac{x + \pi}{2\pi} }.
\]
This is the affine unitary map from $L^2(0, 1)$ onto $L^2(-\pi, \pi)$. We apply Theorem \ref{BBB:th} after the unitary transformation $T = \mathcal F U$ and obtain that the system $\{e^{i\lambda_n t} \mid n\in A\} \cup \{r_n\mid n\notin A\}$ is complete in $L^2(0,1)$.
Let
\[
    M=\clspan \{e^{i\lambda_n t} \mid n\in A\} \qquad \text{and} \qquad N=\clspan \{r_n\mid n\notin A\}.
\]
Biorthogonality gives $M\perp N$, while completeness of the mixed
system gives $L^2(0,1)=M\oplus N$. Then from $f \perp N$ follows $f \in M$. 
\end{proof}

Next we need the following result on the coefficients of a functions with a lacunary spectrum.
\begin{lemma}\label{lem:corollary_ing}
Let $\Lambda=\{\lambda_n\}_{n\in\mathbb Z}\subset\mathbb R$ be a strictly increasing lacunary sequence with lacunarity constants $m^\pm>1$ such that the system $\{ e^{i \lambda t} \}_{\lambda \in \Lambda}$ is minimal. Put $f_j(t)=\sum_{\lambda\in\Lambda}c_{\lambda}^{(j)}e^{i\lambda t}$ with some finite coefficient sequence $\{c_{\lambda}^{(j)}\}_{\lambda\in\Lambda}$. If $f_j\to f \in L^2(0,1)$ for some $f\in L^2(0,1)$, then exists $\{ c_\lambda \} \in \ell^2(\Lambda)$ such that $\{c_{\lambda}^{(j)}\}_{\lambda\in\Lambda} \to \{ c_\lambda \} \in \ell^2(\Lambda)$ as $j \to +\infty$.
\end{lemma}

\begin{proof}
Put $\lambda'_n = \frac{\lambda_n}{2\pi}$ and $d_n^{(j)} = (2\pi)^{-1/2} e^{i\lambda_n/2} c_{\lambda_n}^{(j)}$. Then in the sense of $L^2$-norm we have
\[
 Uf_j = \sum_{n\in\mathbb Z} d_n^{(j)} e^{i\lambda'_n t} \longrightarrow  Uf  \quad\text{ as } j \to \infty.
\]
Denote the left-hand side by $F_j$ and the right-hand side by $F$. By minimality, the transformed system has a biorthogonal family $\{R_n\}_{n\in\mathbb Z}$, so $d_n^{(j)} = (F_j,R_n) \to (F,R_n)$ as $j \to \infty$ for every $n$. Since $\lambda'_{n+1}-\lambda'_n \to\infty$ as $|n|\to\infty$,
choose $N$ so that
\[
    \inf_{\substack{n\ne m\\ |n|,|m|>N}}
    |\lambda'_n-\lambda'_m| > 1.
\]
The finite sums $G_j := \sum_{|n|\le N} d_n^{(j)} e^{i\lambda'_n t}$ converge in $L^2(-\pi,\pi)$. Hence $H_j := F_j-G_j$ is Cauchy sequence. By Theorem~\ref{th:Ing}, for some $C>0$ independent of $j$ and $s$ we have
\[
    \sum_{|n|>N} |d_n^{(j)}-d_n^{(s)}|^2 \leq C \|H_j-H_s\|_{L^2(-\pi,\pi)}^2 \longrightarrow 0
    \quad \text{ as } j,s\to\infty.
\]
Together with convergence of the finitely many remaining coordinates,
this shows that $\{d_n^{(j)}\}_{n\in\mathbb Z}$ is Cauchy in
$\ell^2(\mathbb Z)$ and therefore converges to some
$\{d_n\}_{n\in\mathbb Z}\in\ell^2(\mathbb Z)$.
The conclusion follows with
$c_{\lambda_n}=(2\pi)^{1/2}e^{-i\lambda_n/2}d_n$.
\end{proof}

\subsection{Completeness} 
    By multiplying the initial system by $e^{-i\lambda_0 t}$ for some $\lambda_0 \in \Lambda(A)$ we may assume that $0 \in \Lambda(A)$.
    
     Let us fix an arbitrary $h \in L^2(0, 1)$ such that $ \langle h, wr_n \rangle = \langle wh, r_n \rangle = 0$ for any $n \notin A$. To prove completeness we need to show that $h = 0$. This is equivalent to $f = wh = 0$. Using Lemma \ref{lem:hereditary} we obtain that $f \in \clspan \{ e^{i\lambda t} \mid \lambda \in \Lambda(A)\} $. Hence, there exists a sequence of polynomials $f_j(t) = \sum\limits_{\lambda \in \Lambda(A)} c_{\lambda}^{(j)} e^{i \lambda t}$ such that $f_j \to f$ in $L^2(0, 1)$. Since $\{e^{i \lambda t}\}_{\lambda \in \Lambda(A)}$ is minimal, there exists a limit $ \{ c^{(j)}_\lambda \} \to \{ c_\lambda \}$ in $\ell^2(\Lambda(A))$ by Lemma \ref{lem:corollary_ing}. Let $\widetilde{f}_j$ be natural extention of $f_j$ as an entire function: it belongs $L_{loc}^2(\R)$ by definition. Since $\Lambda(A)$ is lacunary, there is $\gamma > 0$ such that $\inf_{n \in \Z}|\lambda_{n + 1} - \lambda_n| > \gamma$ in monotone numeration of $\Lambda(A)$. Then by Theorem \ref{th:Ing} for all $T > 0$ exists $C(T, \gamma)$ such that 
     \[
     \| \widetilde{f}_j - \widetilde{f}_s \|_{L^2(-T, T)} \leq C(T, \gamma) \left\| \left\{ c_{\lambda}^{(j)} - c^{(s)}_{\lambda}\right\}_{\lambda \in \Lambda} \right\|_{\ell^2(\Lambda(A))}
     \xrightarrow[j, s \to + \infty]{} 0.
     \]
     Hence for all $T> 0$ exists a limit $\widetilde{f}_j \to \widetilde{f} $ in $ L^2(-T, T)$, and so in $ L_{loc}^2(\R)$. As a consequence, $\widetilde{f}_j \to \widetilde{f} $ in $L_{loc}^1(\R)$.
     Also $\frac{f}{w} = h \in L^2(0, 1)$. Then we obtain that there exists $\varepsilon_2 > 0$ such that we have
    \[
    \liminf_{t \to 0^+} \omega^{(1)}_{\widetilde{f}} \lr*{ t } \exp\left(  \frac{1+\varepsilon_2}{2} \left({{1} \over {\ln m^{+}}} + {{1} \over {\ln m^{-}}}\right)\ln^2 t \right) = 0
    \]
    by Lemma \ref{lem:main}. Now one can apply Theorem \ref{belov:th_8}. Indeed, relations \eqref{ineq:A} and \eqref{ineq:B} follows by Lemma \ref{belov:lem_21}. Then $c_{\lambda}^{(j)} \to 0$ as $j \to +\infty$. By Lemma \ref{lem:corollary_ing} $\{c_{\lambda}^{(j)}\}_{\lambda \in \Lambda(A)} \to \{c_{\lambda}\}_{\lambda \in \Lambda(A)}$ in $\ell^2(\Lambda(A))$ for some $\{c_{\lambda}\}_{\lambda \in \Lambda(A)}$. Hence, $\{c_{\lambda}\}_{\lambda \in \Lambda(A)} = 0$ in $\ell^2(\Lambda(A))$. Then by Ingham inequality we have
    \[
    \| f_j \|_{L^2(0, 1)} \leq \| \widetilde{f_j} \|_{L^2(-1, 1)} \leq C(1, \gamma) \left\| \left\{c_{\lambda}^{(j)}\right\}_{\lambda \in \Lambda(A)} \right\|_{\ell^2(\Lambda(A))} \to 0 \quad \text{as} \quad j \to +\infty.
    \]
    On the other hand, $f_j \xrightarrow[]{L^2(0, 1)} f$. Hence, $f = 0$.

\subsection{Non-minimality}
Suppose that the system $\{wr_n\}_{n \in \Z \setminus A}$ is minimal. Then there exists $k \in \Z \setminus A$ and a function $h \in L^2(0, 1)$ such that $(h, wr_n) = 0$ for all $n \notin A' = A \cup \{ k \}$, while $(h, wr_k) = 1$. Since the lacunarity constants of $\Lambda(A')$ and the weight condition for $w(t)$ remain identical, we can conclude that $\{w r_n\}_{n \in \Z \setminus A'}$ is complete. This implies $h = 0$, hence $(h, wr_k) = 0$, which is a contradiction.

\section{Proof of Theorem \ref{th:contr}}
The counterexample is based on Belov’s construction \cite[Theorem 19]{Belov}, which we apply to establish the incompleteness of the weighted exponential system. To keep the presentation self-contained, we reproduce the construction and decay estimates needed for this application.

We choose $A=\{\pm m^{k}\}_{k\in\mathbb{N}_0}\cup\{0\}$. Then $2\pi A = \Lambda(A)$, and the lacunarity constants are equal to $m^{-}=m^{+}=m>1$. We aim to construct a bounded function $h \in L^{2}(0,1) \setminus \{ 0 \}$ such that:
\begin{enumerate}
    \item $h=c_0+\sum\limits_{k=0}^{+\infty}a_{-k}e^{-2 \pi i m^{k}t}+\sum\limits_{k=0}^{+\infty}b_{k}e^{2\pi im^{k}t}$ in the sense of $L^2$-norm,
    \item $h(t)=o\left(\exp\left(-\frac{1}{\ln m}\ln^{2}t + \ln t\right)\right)$ as $t\rightarrow0$.
\end{enumerate}
Taking $w=\exp\lr*{-\frac{1}{\ln m} \ln^2 t + \ln t}$, we then obtain that $\frac{h}{w}\in L^{2}(0,1)$. This implies incompleteness of the system $\{w(t)e^{2\pi int}\}_{n\in\mathbb{Z}\setminus A}$ by Lemma \ref{lem:general}.

\subsection{Canonical product}

Let $M = m^2$. We define an entire function $G$ as follows:
\[
  G(z) = \prod\limits_{n = 1}^\infty \left(   1 + \frac{z}{M^{n - 1}}   \right)  = 1 + \sum\limits_{k = 1}^\infty \left(  \sum\limits_{ 0 < n_1 <\dots < n_k }  \frac{1}{M^{n_1 - 1} \dots M^{n_k - 1}} \right) z^k.
\]
Set $a_0 = 1$ and define the coefficients $a_k$ from the relation $G(z) = 1 + \sum_{k = 1}^\infty a_k z^k$. Then $a_k$ is given by
$$
        a_k = \sum\limits_{ 0 < n_1 <\dots < n_k }  M^{1 - n_1} \dots M^{1 - n_k} = \sum\limits_{\tau_1, \dots, \tau_n = 1}^\infty M^{1 - \tau_1}M^{1 - \tau_1-\tau_2}\dots M^{1 - \sum\limits_{j = 1}^k \tau_j} = $$
$$=M^k \prod\limits_{j = 1}^k   \sum\limits_{\tau_j = 1}^\infty M^{ -(k + 1 - j) \tau_j } = M^k \prod\limits_{j = 1}^k \frac{1}{M^j - 1} = M^{-\frac{ k(k -1) }{ 2 } } \prod\limits_{j = 1}^k \frac{1}{1 - M^{-j}}.
$$
We have $G( -M^{n - 1} ) = 0 $ for $n \geq 1$ and $G\left( -\frac{1}{M} \right) = \prod\limits_{n = 1}^\infty (1 - M^{-n}) $ by the definition of $G$. Hence,
\begin{equation}
\label{eq:Gmod}
        \sum\limits_{k = 0}^\infty (-1)^k a_k M^{k(n - 1)}=0 \quad \text{for all } n \geq 1 \quad \text{ and } \quad
\sum\limits_{k = 0}^\infty (-1)^k a_k M^{-k} = \prod\limits_{n = 1}^\infty (1 - M^{-n}) .
\end{equation}
Since $G$ is positive on $(0, +\infty)$, the function $\ln G(r) = \sum\limits_{n = 1}^\infty g(M^{n - 1})$ is well-defined on $(0, +\infty)$, where $g(t) = \ln\left(1 + \frac{r}{t} \right)$. Let $Y_0(t)$ be counting function of the sequence $\{ M^n \}_{n \in \N_0}$. We have
$$Y_0(t) =  \begin{cases}
1 + \left\lfloor \frac{\ln t}{\ln M} \right\rfloor, & t \geq 1,\\
    0, & \text{otherwise}.
\end{cases}$$
The function $g:(0, +\infty) \to \R$ is absolutely continuous on each finite interval not containing $0$ and $g(t) \to 0$ monotonically as $t \to +\infty$. Then we can use the Stieltjes integral:
$$
        \ln G(r) =  \int\limits_{0}^{+\infty}  -g'(t) Y_0(t) dt = \int\limits_{0}^{+\infty} \frac{r}{t(r + t)} Y_0(t) dt =$$
$$= \int\limits_{1}^{+\infty} \frac{r}{t(r + t)} \left( \frac{1}{2} + \frac{\ln t}{\ln M} \right) dt + \int\limits_{1}^{+\infty} \frac{r}{t(r + t)} \left( \frac{1}{2} - \left\{ \frac{\ln t}{\ln M} \right\} \right) dt = I_1 + I_2.
$$
By a change of variable, we obtain:
\[
I_2 = \int\limits_{0}^{+\infty} \frac{r \ln M}{ (r + M^t)} \left( \frac{1}{2} - \left\{ t \right\} \right) dt.
\]
Then $0 \leq I_2 \leq \frac{r \ln M}{8(r + 1)} \leq \frac{\ln M}{8}$. This is a standard estimate, which can be found, for example, in Lemma 6 of \cite{Belov}. Now let us estimate $I_1$:
$$ I_1 = \int\limits_{1}^{+\infty} \frac{r}{t(r + t)} \left( \frac{1}{2} + \frac{\ln t}{\ln M} \right) dt =  \int\limits_{1/r}^{+\infty} \frac{1 }{u (1 + u  )} \left( \frac{1}{2} + \frac{\ln r}{\ln M} + \frac{\ln u}{\ln M}  \right) du = $$
$$      =  \lr*{ \frac{1}{2} + \frac{\ln r}{\ln M} } \ln \lr*{r + 1} + \frac{1}{\ln M} \int\limits_{1/r}^{+ \infty} \frac{\ln u}{u(1 + u )} du = $$
$$  =  \lr*{ \frac{1}{2} + \frac{\ln r}{\ln M} } \ln \lr*{r + 1} +   \frac{1}{\ln M} \int\limits_1^{+ \infty} \frac{\ln u}{u(1 + u )} du  - \frac{\ln^2 r}{2 \ln M} - \frac{1}{\ln M} \int\limits_{1/r}^{1} \frac{\ln u}{1 + u} du.
$$
Combining these estimates, for $r \ge 1/m$ we have:
$$
    \left| \ln G(r) - \frac{\ln^2 r}{2\ln M }  - \frac{\ln r}{2}  \right| =  \left| I_1+ I_2 - \frac{\ln^2 r}{2\ln M }  - \frac{\ln r}{2}  \right| =$$
$$= \left|I_2 + \left( \frac{1}{2} + \frac{\ln r}{\ln M} \right) \ln\left(1 + \frac{1}{r}\right) + \frac{1}{\ln M} \int_1^{+\infty} \frac{\ln u}{u(1+u)} du - \frac{1}{\ln M} \int_{1/r}^1 \frac{\ln u}{1+u} du  \right| \leq $$
$$\leq \frac{\ln M}{8} + \frac{1}{\ln M} \max_{r \ge 1/m} \frac{\ln r}{r} + C_1 + \frac{\pi^2}{12 \ln M}.
$$
Therefore, there exists $C>0$ such that for all $r \ge 1/m$ the following inequality holds:
\begin{equation}\label{ineq:logG} 
|\ln G(r)| \leq  \left| \frac{\ln^2 r}{4\ln m }  + \frac{\ln r}{2} + C \right|.
\end{equation}

\subsection{Construction of the function}

Now we introduce the function $h$ as follows:
$$h(t)=G(-M^{-1})-\sum_{k=0}^{\infty}(-1)^{k}a_{k}M^{-k}\cos(2\pi m^{k}t).$$
Observe that $a_0 = 1$ and so $h \neq 0$. The function $h$ can be differentiated term-by-term, since $a_k = O\lr*{M^{-Ck^2}}$ and differentiation increases the coefficients by at most $O(m^k)$. Consequently, $h \in C^\infty[0, 1] \subset L^2(0, 1)$ and we have
$$h^{(n)}(t) = \begin{cases} \pm\sum\limits_{k=0}^{\infty}(-1)^{k}a_{k}M^{-k}(2\pi m^{k})^{2p+1}\sin(2\pi m^{k}t) , & n = 2p+1,\\
-(-1)^{p}\sum\limits_{k=0}^{\infty}(-1)^{k}a_{k}M^{-k}(2\pi m^{k})^{2p}\cos(2\pi m^{k}t), & n = 2p>0.
\end{cases}$$
Now using \eqref{eq:Gmod} we obtain $h^{(n)}(0)=0$ for all $n > 0$ and $h(0)=G(-M^{-1})-\sum_{k=0}^{\infty}(-1)^{k}a_{k}M^{-k}=0$. Hence, for any $n \in \N$ there exists $\theta = \theta(n, t) \in (0, 1)$ such that
$$h(t) = \frac{t^n}{n!} h^{(n)}(\theta t)  \quad \text{ for all } t \in \R.$$
We estimate the derivative:
$$|h^{(n)}(t)|\leq \begin{cases} (2\pi)^n \sum\limits_{k=0}^{\infty}a_{k}M^{k(p-1/2)}, & n = 2p+1\\
(2\pi)^n \sum\limits_{k=0}^{\infty}a_{k}M^{k(p-1)}, & n = 2p,
\end{cases} = (2\pi)^n G(m^{n-2}).$$

Combining everything together, we obtain:
\begin{equation}\label{ineq:h}
   |h(t)| \leq  \inf\limits_{n \in \N} \frac{G(m^{n-2})}{n!} |2\pi t|^n \text{ for all } t \in \R.
\end{equation}

\subsection{Rate of decay}
Let us use the estimate \eqref{ineq:logG}. There exists $C_1 > 0$ such that
$$
|\ln G(m^{n-2})| \leq \left| \frac{\ln^2(m^{n-2})}{4 \ln m} + \frac{\ln(m^{n-2})}{2} + C \right| = \left| \frac{\ln m}{4} n^2 - \frac{\ln m}{2} n + C_1 \right|.
$$

Using Stirling's formula, we have
$$
\frac{1}{n!} \le \frac{1}{\sqrt{2\pi}} n^{-n-1/2} e^n = \frac{1}{\sqrt{2\pi}} \exp\lr*{-n \ln n - \frac{1}{2} \ln n + n}.
$$

Finally, using the inequality \eqref{ineq:h} and the relation $\inf_n e^{x_n} = e^{\inf_n x_n}$, we obtain
$$
|h(t)| \le C_2 \exp \inf_{n \in \N} \lr*{ \frac{\ln m}{4} n^2 - \frac{1}{2} \ln n - n \ln n - n \lr*{ \frac{\ln m}{2} - \ln|2\pi t| - 1 } }.
$$
\begin{lemma}[Lemma 7 in \cite{Belov}] \label{belov:7}
    For any $\alpha > 0$, the function
    \[
    \varphi(u) = \inf\limits_{n \geq 1} \left(  \alpha n^2 - n \ln n - nu - \frac{1}{2} \ln n  \right)
    \]
    is defined on $\R$, and there exist constants $t_0, C$, depending only on $\alpha$, such that for all $t > t_0:$
    \[
    \left|  \varphi(t) + \frac{t^2}{4\alpha} + \frac{t \ln t}{2\alpha} - \frac{\ln 2\alpha}{2 \alpha} t + \frac{\ln^2 t}{4\alpha} + \left(  \frac{1}{2} + \frac{1}{2\alpha} - \frac{\ln 2\alpha}{2\alpha}  \right) \ln t   \right| \leq C.
    \]
\end{lemma}

Applying Lemma \ref{belov:7} with $\alpha = \frac{\ln m}{4}$ and $u = \frac{\ln m}{2} - \ln|2\pi t| - 1$, we obtain that for sufficiently large $u$ we have
\begin{gather*}
|h(t)| \leq C_3 \exp\lr*{-\lr*{\frac{u^2}{4\alpha} + \frac{u \ln u}{2\alpha} - \frac{\ln 2\alpha}{2\alpha}u + \frac{\ln^2 u}{4\alpha} + \lr*{\frac{1}{2} + \frac{1}{2\alpha} - \frac{\ln 2\alpha}{2\alpha}}\ln u} + C_1} = \\
= C_3 \exp\lr*{-\frac{\lr*{u+\ln u}^2}{\ln m} + \frac{2\ln\lr*{\frac{\ln m}{2}}}{\ln m}u - \lr*{\frac{1}{2} + \frac{2}{\ln m} - \frac{2\ln\lr*{\frac{\ln m}{2}}}{\ln m}}\ln u + C_1} = \\
= C_3 \exp\Biggl(-\frac{1}{\ln m}\ln^2\lr*{\frac{1}{t}\ln\frac{1}{t}} +  \lr*{\frac{2\ln 2\pi}{\ln m} + \frac{2\ln\lr*{\frac{\ln m}{2}}}{\ln m} + \frac{2}{\ln m} - 1}\ln\frac{1}{t} - \\
- \lr*{\frac{3}{2} - \frac{2\ln 2\pi}{\ln m} - \frac{2\ln\lr*{\frac{\ln m}{2}}}{\ln m}}\ln\ln\frac{1}{t} + C_1 + O(1)\Biggr).
\end{gather*}
Then for sufficiently small $t>0$ we have $|h(t)|  \leq \const \cdot e^{-\psi(1/t)}$, where
$$
\psi(x) = \frac{1}{\ln m} \ln^2(x \ln x) - \lr*{ \frac{2\ln 2\pi}{\ln m} +\frac{2\ln\lr*{\frac{\ln m}{2}}}{\ln m} + \frac{2}{\ln m} - 1 } \ln x +
$$
$$
+ \lr*{ \frac{3}{2} -\frac{2 \ln 2\pi}{ \ln m } - \frac{2\ln\lr*{\frac{ \ln m}{2}}}{\ln m }} \ln \ln x.
$$

Note that $\ln t = o(\ln t \ln (- \ln t))$ as $t \to 0^+$, and therefore we have
$$
\frac{|h(t)|}{\exp\lr*{-\frac{1}{\ln m} \ln^2 t + \ln t}} \le \const \cdot \frac{\exp\lr*{-\frac{1}{\ln m} \ln^2 t + \frac{2}{\ln m} \ln t \ln(-\ln t) + O(\ln t)}}{\exp\lr*{-\frac{1}{\ln m} \ln^2 t + \ln t}} \longrightarrow 0.
$$
\section{Proof of Theorem \ref{th:fin_schur}}
To prove the theorem, we need some weill-known combinatorial facts.
\subsection{Definitions}
First we remind some definitions and notations.
\begin{definition} \label{Young:tableau}
        For a fixed integer $n \in \N$ and a partition $\mu = (1 \leq \mu_1 \leq \dots \leq \mu_n)$, a semistandard Young tableau $T$ of shape $\mu$ is a Young diagram (in English notation), where the row lengths from top to bottom are $\mu_n, \mu_{n-1}, \dots, \mu_1$, filled with numbers from $1$ to $K$ (allowing repetitions) such that the entries are non-decreasing along rows and strictly increasing down columns. \\
        We denote the set of all such tableaux by $SSYT_{K, \mu}$. When $K = n$, we simply write $SSYT_\mu$ instead of $SSYT_{n, \mu}$. For each $j$, the weight $w_j(T)$ is defined as the number of times $j$ appears in $T$.
\end{definition}
Throughout this paper all partitions are nondecreasing. Whenever a standard algebraic or combinatorial construction uses nonincreasing order, the components are read in reverse order $(\mu_n,\dots,\mu_1)$. We evaluate $s_\mu$ on $n$ variables, where $n$ is the length of $\mu$.
\begin{definition}
    Let $\mu = (1 \leq \mu_1 \leq \dots \leq \mu_n)$ be a partition. Define the Schur polynomial
        \[
        s_\mu(x_1, \dots, x_n) \eqdef \frac{1}{\prod\limits_{1 \leq j < k \leq n} (x_j - x_k) } \cdot \det \begin{pmatrix}
            x_1^{\mu_n + n - 1} & x_2^{\mu_n + n - 1} & \dots & x_n^{\mu_n + n - 1} \\
            x_1^{\mu_{n - 1} + n - 2} & x_2^{\mu_{n - 1} + n - 2} & \dots & x_n^{\mu_{n - 1} + n - 2} \\
            \vdots & \vdots & \ddots & \vdots \\
            x_1^{\mu_1} & x_2^{\mu_1} & \dots & x_n^{\mu_1}
        \end{pmatrix}.
        \]

\end{definition}
Note that the Schur polynomial is well-defined at coincident points since the numerator divisible by the Vandermonde determinant.
\begin{lemma}
    Following \cite{Macdonald} we use an equivalent definition of the Schur polynomial:
        \[
        s_\mu(x_1, \dots, x_n) \eqdef \sum\limits_{T \in SSYT_\mu} x^T = \sum\limits_{T \in SSYT_\mu} x_1^{w_1(T)} \dots x_n^{w_n(T)}. 
        \]
\end{lemma}
\begin{definition}
    For distinct $x_1, \ldots, x_n$ the divided differences are defined as:
    \begin{enumerate}
        \item $f[x_j] = f(x_j)$ for $1 \leq j \leq n$,
        \item $f[x_k, \dots, x_j] = \frac{f[x_{k+1}, \dots, x_j] - f[x_k, \dots, x_{j - 1}]}{x_j - x_k}$ for $1 \leq k < j \leq n$.
    \end{enumerate}
\end{definition}

\subsection{Auxiliary facts}
First we note that for finite $A$ and for a system $\{ e^{2\pi i n t} \}_{n \in \Z}$ the statement of Lemma \ref{lem:general} can be extended to a weight $w \in L^2(0, 1)$.
    \begin{lemma} \label{lem:L2_weight}
    Fix a set $A = \{ k_1, \dots, k_m \}$ and an arbitrary weight $w \in L^2(0, 1)$. Then 
    \begin{enumerate}
        \item $\{ w(t) e^{2\pi i nt} \}_{n \in \Z \setminus A}$ is complete in $L^2(0, 1)$ if and only if for any $\{c_{m}\}_{m \in A}$ we have
        \[
            \frac{\sum\limits_{ m \in A } c_{m} e^{2 \pi i m t} }{ w(t) } \notin L^2(0, 1) \setminus \{0\}.
        \]
        
        \item $\{ w(t) e^{2\pi i nt} \}_{n \in \Z \setminus A}$ is minimal in $L^2(0, 1)$ if and only if for any $n \notin A$ there exist $\{\xi_{n, m}\}_{ m \in A}$ such that
        \[
            \frac{e^{2 \pi i n t} + \sum\limits_{ m \in A } \xi_{n, m} e^{2 \pi i m t}}{ w(t) } \in L^2(0, 1).
        \]

    \end{enumerate} 
\end{lemma}

\begin{proof}
    Similar to the proof of Lemma \ref{lem:general}, $\{ w(t) e^{2\pi i nt} \}_{n \in \Z \setminus A}$ is complete if and only if for any $f \in L^2(0, 1)$ the condition $\lrsc*{f(t), w(t) e^{2\pi i n t}} = 0$ for all $n \notin A$ implies $f = 0$.
    Fix such an $f \in L^2(0,1)$. By assumption, $fw \in L^1(0, 1)$. Since $w$ is real-valued, the Fourier coefficients of $fw$ can be expressed as $\widehat{fw}(n) = \lrsc*{f(t), w(t) e^{2\pi i n t}}$. Then the condition $\lrsc*{f, wr_n} = 0$ for all $n \notin A$ is equivalent to the fact that $\widehat{fw}(n) = 0$ for $n \notin A$. Using uniqueness of Fourier coefficients for $L^1(0, 1)$, we conclude that $fw$ can be represented as a sum $f(t)w(t) = \sum\limits_{ m \in A } c_{m} e^{2 \pi i m t}$, which belongs to $L^2(0, 1)$ because of finiteness. Then we obtain
    $$f(t) =  \frac{\sum\limits_{ m \in A } c_{m} e^{2\pi i m t} }{ w(t) }, $$
    where the equality is understood in the sense of $L^2(0, 1)$. Thus, the system $\{ w(t) e^{2\pi i nt} \}_{n \in \Z \setminus A}$ is complete if and only if for any $f = \sum\limits_{ m \in A } c_{m} e^{2 \pi i m t}$ in $L^2(0, 1)$ the following implication holds:
    \[
    \frac{f }{ w } \in L^2(0, 1)  \Longrightarrow f = 0 \text{ almost everywhere}.
    \]
The minimality of $\{ w(t) e^{2\pi i nt} \}_{n \in \Z \setminus A}$ is equivalent to existence of its biorthogonal system $\{f_n  \}_{n \in \Z \setminus A} \subset L^2(0, 1)$. Similarly to the previous point and Lemma \ref{lem:general}, this is equivalent to the fact that $f_n$ has the representation
    \[
    f_n(t) = \frac{e^{2 \pi i n t} + \sum\limits_{ m \in A } \xi_{n, m} e^{2 \pi i m t}}{ w(t) },
    \]
where the fraction on the right-hand side of the equality belongs to $L^2(0,1)$, as required.
\end{proof}
We now demonstrate some properties of the zero set of the weight we have defined.
\begin{lemma}\label{lem:proving_def}
If $P/w\in L^2(0,1)$ for some non-zero trigonometric polynomial $P$, then the zero set $Z(w)$ from Definition \ref{def:zero_set} exists and is unique. Moreover:
\begin{enumerate}
    \item the non-integrability in \eqref{eq:defzero2} is local at $t_d$,
    \item for every trigonometric polynomial $Q$
    \[
    \frac{Q}{w}\in L^2(0,1)
    \quad\Longleftrightarrow\quad
    Q \text{ vanishes at every }t_j\in Z(w)\text{ to order at least }a_j.
    \]
\end{enumerate}
\end{lemma}

\begin{proof}
Let $t_1,\dots,t_N$ be the distinct zeros of $P$ on $\mathbb T$. We keep each $t_j$ with the smallest nonnegative multiplicity $a_j$, for which the function $\frac{\lr*{e^{2\pi i t}-e^{2\pi i t_j}}^{a_j}}{w(t)}$ is square-integrable near $t_j$. Such a multiplicity exists, since taking the order of $P$ at $t_j$ suffices. Points with $a_j=0$ are omitted. All neighborhoods are taken in $\mathbb T$.

Let $Z$ be the resulting multiset. Near each $t_j$, the other factors of $D_Z$ are bounded and bounded away from zero, so $D_Z/w$ is square-integrable there. Outside small neighborhoods of these points, $D_Z/P$ is bounded. Hence $D_Z/w\in L^2(0,1)$. Note that removing one copy of a point with $a_j>0$ destroys square-integrability near that point, by the choice of $a_j$. It preserves square-integrability outside any neighborhood of that point, where the reciprocal of the removed factor is bounded. Thus $Z$ satisfies conditions \eqref{eq:defzero1} and \eqref{eq:defzero2}, and part 1 holds.

If $Q$ vanishes to the required order at every point of $Z$, then $Q/D_Z$ extends to a bounded function on $\mathbb T$. Hence, $Q/w \in L^2(0, 1)$.
Conversely, if $Q/w \in L^2(0, 1)$, near $t_j$ a non-zero polynomial $Q$ differs from its vanishing factor by a function bounded and bounded away from zero. Now condition $Q/w\in L^2(0,1)$ implies that its order of vanishing at least $a_j$, because of the choice of $a_j$. This proves part 2.

Finally, if $Z'$ also satisfies Definition \ref{def:zero_set}, part 2 applied to $D_{Z'}$ implies that $Z'$ contains $Z$ (counting multiplicities). Any additional copy of a point could be deleted while leaving all the zeros required by $Z$. By part 2, this would contradict \eqref{eq:defzero2} for $Z'$. Hence $Z'=Z$.
\end{proof}

\begin{remark} \label{lem:pereform}
    Under the conditions of Lemma \ref{lem:L2_weight}, we have the following reformulation:
    \begin{enumerate}
        \item $\{ w(t) e^{2\pi i nt} \}_{n \in \Z \setminus A}$ is complete in $L^2(0, 1)$ if and only if for any non-trivial $\{c_{j}\}_{1 \leq j \leq m} \subset \C$ the multiset Z is undefined or the following property holds:
        \[
            \text{The zero set of the function } \sum\limits_{ 1 \leq j \leq m } c_{j} e^{2\pi i k_j t} \text{ does not contain $Z$ as multiset}.
        \]
        
        \item $\{ w(t) e^{2\pi i nt} \}_{n \in \Z \setminus A}$ is minimal in $L^2(0, 1)$ if and only if for any $n \notin A$ the multiset Z is well-defined and there exist $\{\xi_{n, j}\}_{1 \leq j \leq m} \subset \C$ such that
        \[
            \text{The zero set of } e^{2\pi i n t} + \sum\limits_{ 1 \leq j \leq m } \xi_{n, j} e^{2 \pi i k_j t} \text{ contains $Z$ as multiset}.
        \]
    \end{enumerate}
    \end{remark}

\begin{definition}
    Let $A$ be a finite set and let $w$ a weight with zero multiset $Z$. Let $t_1,\ldots,t_N$ denote the distinct nodes of $Z$ with multiplicities $a_1,\ldots,a_N$ respectively. For each $n \in \Z$ define the vector
    \[
    v_n =\lr*{ e^{2 \pi i n t_1 }, {n} e^{2 \pi i n t_1 }, \dots, {n}^{a_1 - 1} e^{2 \pi i n t_1 }, e^{2 \pi i n t_2 }, \dots, {n}^{a_N - 1} e^{2 \pi i n t_N } }.
    \]
    We also define the matrix $M_{w, A}$, whose columns are the vectors $v_{k_j}$:
    \[
    M_{w, A} =  \begin{pmatrix} v_{k_1} & v_{k_2} & \dots & v_{k_m} \end{pmatrix}.
    \]

\end{definition}

Now we can state the following lemma:
\begin{lemma} \label{lem:power_Z}
    Let $w \in L^2(0, 1)$ be a weight and $A = \{ k_1, \dots, k_m\}$. If the system $\{ w(t) e^{2\pi i nt} \}_{n \in \Z \setminus A}$ is complete and minimal, then $Z$ is well-defined and $|Z| = |A|$ counting multiplicities.
\end{lemma}

\begin{proof}
    By Remark \ref{lem:pereform} the set $Z$ is well-defined. The statement that the zero set of a function $f$ covers the point $t_j$ with multiplicity $a_j$ is equivalent to the condition
    \[
    \sum\limits_{ 1 \leq n \leq m } c_{n} k _n^{s} e^{2\pi i k_n t_j} = 0 \quad \text{for any} \quad s \in \{0, \dots, a_j - 1\},
    \]
    where the condition that the derivatives of the sum vanish is simplified by dividing all the equations by $(2\pi i)^s$. Then conditions of Remark \ref{lem:pereform} is equivalent to the absence of non-trivial $c_1,\dots,c_m$ such that
\begin{equation} \label{matrix:comp}
M_{w, A}
 \cdot 
 \begin{pmatrix}
     c_1 \\
     c_2 \\
     \vdots \\
     c_m
 \end{pmatrix}
 = 0.
\end{equation}

Similarly, minimality is equivalent to the fact that for any $n \notin A$ there exist $\xi_{n, 1}, \dots, \xi_{n, m}$ such that
\begin{equation} \label{matrix:min}
M_{w, A}
 \cdot 
 \begin{pmatrix}
     \xi_{n, 1} \\
     \xi_{n, 2} \\
     \vdots \\
     \xi_{n, m}
 \end{pmatrix}
 = -v_n.
\end{equation}
Consider $\CC^{|Z|}$ and $V_A = \spann\{  v_{k_j} \ | \ 1 \leq j \leq m \} \leq \CC^{|Z|}$. It is clear that $\dim V_A \leq m$. Using \eqref{matrix:comp} we have  $\dim V_A \geq m $. Hence, $\dim V_A = m$. Moreover, \eqref{matrix:min} is equivalent to the fact that $v_n \in \spann \{ v_{k_j} \}_{1 \leq j \leq m} = V_A$ for any $n \notin A$. Thus $v_n \in V_A$ for all $n \in \Z$ since $v_n \in V_A$ by definition $V_A$. In particular, $V_{|Z|} = \spann \{v_1, \dots, v_{|Z|}\} \subseteq V_A$. But the matrix with columns $(v_1, \dots, v_{|Z|})$ is a confluent Vandermonde matrix, which has a non-zero determinant. Indeed, a direct computation gives
\[
\det(v_1,\dots,v_{|Z|}) = \prod_{d=1}^N \lr*{e^{{\pi i t_d}{a_d(a_d+1)}} \prod_{q=0}^{a_d - 1}q!} \prod_{1\leq l<s\leq N}(e^{2\pi i t_s}-e^{2\pi i t_l})^{a_s a_l} \neq 0.
\]
Now $\dim V_{|Z|} = |Z|$. Consequently, $|Z| = \dim V_A = m$ as required.
\end{proof}
Now we can state a criterion for the completeness and minimality of the system $\{ w(t)e^{2 \pi i n t} \}_{n \in \Z \setminus A}$ for $|A| < \infty$.
\begin{lemma} \label{lem:criteria}
    Let $w$ be a weight and $A \subset \Z$ be finite. The system $\{ w(t) e^{2\pi i nt} \}_{n \in \Z \setminus A}$ is complete and minimal if and only if 
    \begin{enumerate}
        \item $w \in L^2(0, 1)$,
        \item $w$ has a finite zero set $Z$ and $|Z| = |A|$,
        \item $ \det M_{w, A} \neq 0$.
    \end{enumerate}
\end{lemma}

\begin{proof} Suppose the system is exact. Then $|w(t) e^{2\pi i n t}| = |w(t)|$ for all $n \in \Z$. Since $w(t) e^{2\pi i n t} \in L^2(0, 1)$ for $n \notin A$, we can conclude that $w \in L^2(0, 1)$. We have $|Z| = |A|$ by Lemma \ref{lem:power_Z}. Finally, $\det M_{w, A} \neq 0$ since the full rank of this matrix is equivalent to completeness from \eqref{matrix:comp}. \par 
For the proof of sufficiency note that completeness is equivalent to $M_{w, A}$ having full rank. Since $M_{w, A} \in \mathcal{M}_{|Z| \times |A|}(\C)$ and $|Z| = |A|$, this is equivalent to $ \det M_{w, A} \neq 0$, which means that the system $\{ w(t) e^{2\pi i nt} \}_{n \in \Z \setminus A}$ is complete. To establish minimality, observe that $|Z| = |A|$ together with completeness implies $\dim V_A = m = |Z|$. Thus it spans $\C^{|Z|}$. In particular, it contains all $v_n$. This is equivalent to \eqref{matrix:min} by the argument from the previous lemma.
\end{proof}
\subsection{Proof of the theorem}
Shifting $A$ by an integer constant does not affect completeness and minimality of the system $\{ w(t) e^{2\pi i nt} \}_{n \in \Z \setminus A}$. Thus one can assume that $k_1 = 1$, meaning that $A' = ( k_j - j + 1 )_{1 \leq j \leq m}$. It suffices to verify that the conditions of the theorem are equivalent to those of Lemma \ref{lem:criteria}. It remains only to prove that $\det M_{w,A}\neq 0$ if and only if $s_{A'}(e^{2\pi i t_1},\dots,e^{2\pi i t_m})\neq 0$, since the remaining conditions coincide.

Suppose first that all roots are simple, and set $z_r=e^{2\pi i t_r}$. In this case $M_{w,A}=(z_r^{k_j})_{r,j=1}^m$ and the alternant formula gives
\[
    \det M_{w,A} =s_{A'}(z_1,\dots,z_m) \prod_{1\leq r<s\leq m}(z_s-z_r).
\]
Since the nodes are pairwise distinct, the Vandermonde factor is nonzero. Therefore $\det M_{w,A}\ne0$ if and only if $s_{A'}(z_1,\ldots,z_m)\ne0$.

We now allow multiple roots. Let $z_l=e^{2\pi i t_l}$ for $1\leq l\leq N$ be the distinct roots. Let their multiplicities be $a_1,\dots,a_N$, where $a_1+\cdots+a_N=m$. Set $f_j(x)=e^{2\pi i k_jx}$. For each $l$, replace $t_l$ temporarily by points $x_{l,0},\dots,x_{l,a_l-1}$ chosen so that all the numbers $z_{l,q}:=e^{2\pi i x_{l,q}}$ are pairwise distinct. We order the pairs $(l,q)$ lexicographically. For $x = (x_{l, q})$ we define
\[
    M(x)=\lr*{f_j(x_{l,q})}_{(l,q),j} \qquad \text{and} \qquad M_d(x) = \lr*{f_j[x_{l,0},\dots,x_{l,q}]}_{(l,q),j}.
\]
For a fixed $l$ the divided difference in the $q$-th row of $M_d(x)$ is a linear combination of the first $q+1$ rows in the corresponding block of $M(x)$. Thus the passage from $M(x)$ to $M_d(x)$ is effected by a block diagonal matrix whose blocks are lower triangular. Indeed, the explicit formula for a divided difference shows that the coefficient of $f_j(x_{l,q})$ in $f_j[x_{l,0},\dots,x_{l,q}]$ is $\prod\limits_{p=0}^{q-1}(x_{l,q}-x_{l,p})^{-1}$. Consequently,
\[
    \det M_d(x) =\frac{\det M(x)}{\prod\limits_{l=1}^N\prod\limits_{0\leq p<q<a_l}(x_{l,q}-x_{l,p})}.
\]
Applying the alternant formula to $M(x)$ and separating within-block and between-block factors of its Vandermonde determinant, we obtain
\[
    \det M_d(x) = s_{A'}\lr*{(z_{l,q})_{(l,q)}} \lr*{\prod_{l=1}^N\prod_{0\leq p<q<a_l} \frac{z_{l,q}-z_{l,p}}{x_{l,q}-x_{l,p}}} \lr*{\prod_{1\leq l<s\leq N} \prod_{p=0}^{a_l-1}\prod_{q=0}^{a_s-1} (z_{s,q}-z_{l,p})}.
\]
Each quotient $(z_{l,q}-z_{l,p})/(x_{l,q}-x_{l,p})$ extends analytically as function over $x_{l,q}, x_{l,p}$ across $x_{l,q} = x_{l,p}$, with value $2\pi i e^{2\pi i x_{l,p}}$. Consequently, the right-hand side (and hence the equality) extends continuously across coincident variables. Then let $x_{l,q}\to t_l$ for every $l$ and $q$. Since $f_j$ is analytic, the elementary confluence formula for divided differences gives
\[
    f_j[x_{l,0},\dots,x_{l,q}] \longrightarrow \frac{f_j^{(q)}(t_l)}{q!} =\frac{(2\pi i)^q}{q!}\,k_j^qz_l^{k_j}.
\]
This follows directly by substituting the Taylor expansion of $f_j$ at
$t_l$ into the divided difference. By the definition of $M_{w,A}$, it
follows that
\[
    \lim_{x \to t}\det M_d(x) =\lr*{\prod_{l=1}^N\prod_{q=0}^{a_l-1} \frac{(2\pi i)^q}{q!}}\det M_{w,A}.
\]

On the other hand, each within-block quotient tends to $2\pi i z_l$,
and there are $\binom{a_l}{2}$ such quotients in the $l$-th block. Each
between-block factor tends to $z_s-z_l$ and occurs $a_la_s$ times.
Therefore,
\[
    \lim_{x \to t}\det M_d(x) = s_{A'}\lr*{ z_1,\dots,z_1,\dots,z_N} \prod_{l=1}^N(2\pi i z_l)^{\binom{a_l}{2}} \prod_{1\leq l<s\leq N}(z_s-z_l)^{a_l a_s}, 
\]
where $z_l$ repeats $a_l$ times. Comparing the two expressions for the limit and using $\sum_{q=0}^{a_l-1}q=\binom{a_l}{2}$, we obtain the confluent determinant identity
\[
    \det M_{w,A} =s_{A'}\lr*{z_1,\dots,z_1,\dots,z_N} 
    \prod_{l=1}^N \lr*{z_l^{\binom{a_l}{2}} \prod_{q=0}^{a_l-1}q!} \prod_{1\leq l<s\leq N}(z_s-z_l)^{a_la_s}.
\]
The factors following the Schur polynomial are nonzero because each $z_l$ is nonzero and the $z_l$ are pairwise distinct. Hence, $\det M_{w,A}\neq0$ if and only if the Schur polynomial evaluated at the roots (repeated according to their multiplicities) is nonzero.

\subsection{Proof of Corollary \ref{cor:schur}}
First we prove a technical lemma.
\begin{lemma} \label{lem:conv}
For every $T\in SSYT_\mu$, its weight vector
$(w_1(T),\ldots,w_n(T))$ belongs to
\[
 \conv\{(\mu_{\sigma(1)},\ldots,\mu_{\sigma(n)}) \mid \sigma\in S_n\}.
\]
\end{lemma}

\begin{proof}
The entries in each column of a semistandard tableau are distinct. Hence, for every set $S\subset\{1,\ldots,n\}$ with $|S|=r$, the number of entries of $T$ that belong to $S$ is at most
\[
 \sum_{\text{columns }C}\min(r,\operatorname{height}C)
 =\sum_{j=n-r+1}^n\mu_j.
\]
Also $\sum_{j=1}^n w_j(T)=\sum_{j=1}^n\mu_j$. Now, the subset-sum characterization of majorization completes the proof.
\end{proof}
As before, we assume $k_1 = 1$. Using Theorem \ref{th:fin_schur}, it suffices to verify that $s_{A'} (e^{2 \pi i t_1}, \dots, e^{2\pi it_m}) \neq 0$, where $A' = (k'_j)_{1 \leq j \leq m} = (k_j -j + 1)_{1 \leq j \leq m}$ and $k_1 = 1.$ Furthermore, since we only need to verify that $s_{A'} (e^{2 \pi i t_1}, \dots, e^{2\pi it_m}) \neq 0$, we can consider $\{t_j\}_{1 \leq j \leq m}$ simply as a multiset in $\R$. Since Schur polynomials are homogeneous, their non-vanishing property is preserved if we multiply all arguments by $e^{-2\pi i \gamma}$ for any $\gamma \in \R$. This corresponds to a global shift $t'_j = t_j - \gamma$. We can choose $\gamma$ such that the shifted points $t'_j$ possess representatives $\tilde{t}_j \in \R$ of their corresponding equivalence classes in $\R/\Z$ whose linear diameter $\max_{i, j} |\tilde{t}_i - \tilde{t}_j|$ is equal to $\diam Z$. Without loss of generality, we will assume that our initial $t_j$ are exactly these representatives. 

For any $x \in \R$ the condition $s_{A'} (e^{2 \pi i t_1}, \dots, e^{2\pi it_m}) \neq 0$ is equivalent to $s_{A'} (e^{2 \pi i \theta_1}, \dots, e^{2\pi i\theta_m}) \neq 0$, where $\theta_j = t_j - x$. This means that
\[
\sum\limits_{T \in SSYT_{A'}} e^{2\pi i \theta_1 w_1(T)}\dots e^{2\pi i \theta_m w_m(T)}  = \sum\limits_{T \in SSYT_{A'}} \exp\lr*{\sum\limits_{j = 1}^m 2\pi i \theta_j w_j(T)} \neq 0.
\]
To show that this holds under our assumptions, we will demonstrate that the real part of each term is greater than zero, that is:
\[
\Re \exp\lr*{\sum\limits_{j = 1}^m 2\pi i \theta_j w_j(T)} = \cos \lr*{\sum\limits_{j = 1}^m 2\pi  \theta_j w_j(T)} > 0 
\ \ \textit{for all} \ \ T \in SSYT_{A'}.
\]
It suffices to show that for some $x \in \R$ we have
\[
\left| \sum\limits_{j = 1}^m \theta_j w_j(T) \right| < \frac{1}{4} \quad \textit{for all} \quad T \in SSYT_{A'}.
\]
The sum inside the absolute value can be interpreted as function $f(\theta_1, \dots, \theta_m,  w_1, \dots, w_m)$. This function is linear with respect to the vector $(w_1,\dots, w_m)$, and moreover $(w_1,\dots, w_m)$ lies in the convex hull $\conv \left\{ (k'_{\sigma(1)}, \dots, k'_{\sigma(m)}) \ | \ \sigma \in S_m \right\}$ by Lemma \ref{lem:conv}. Hence, 
\[
|f(\theta_1, \dots, \theta_m, w_1, \dots, w_m)| \leq \max_{\sigma \in S_m} \left| \sum\limits_{j = 1}^m \theta_j k'_{\sigma(j)} \right|.
\]
Without loss of generality, we assume that $\theta_1 \leq \dots \leq \theta_m.$ Then we have
\[
\max_{\sigma \in S_m} \sum\limits_{j = 1}^m \theta_j k'_{\sigma(j)} = \sum\limits_{j = 1}^m \theta_j k'_{j} \qquad \text{and} \qquad \min_{\sigma \in S_m} \sum\limits_{j = 1}^m \theta_j k'_{\sigma(j)} = \sum\limits_{j = 1}^m \theta_j k'_{m - j + 1},
\]
since $0 < k'_1 \leq \dots \leq k'_m$. Let us denote $B_1 = \sum\limits_{j = 1}^m t_j k'_{j}, \ B_2 = \sum\limits_{j = 1}^m t_j k'_{m - j + 1}$ and $B = \sum\limits_{j = 1}^m k'_j$. Then, from the above observations and the fact that $\theta_j = t_j - x$, we obtain that for any $x \in \R$ the following inequality holds:
\begin{gather*}
    \left|\sum\limits_{j = 1}^m \theta_j w_j(T)\right| \leq \max \lr*{ \left| \min_{\sigma \in S_m} \sum\limits_{j = 1}^m \theta_j k'_{\sigma(j)} \right|, \left| \max_{\sigma \in S_m} \sum\limits_{j = 1}^m \theta_j k'_{\sigma(j)} \right| } = \\ = \max \lr*{ \left| \sum\limits_{j = 1}^m \theta_j k'_{j} \right|, \left| \sum\limits_{j = 1}^m \theta_j k'_{m - j + 1} \right| } = \max \lr*{ |B_1 - xB|, |B_2 - xB|  }.
\end{gather*}
Since this holds for any $x$, we can conclude that
\[
\left|\sum\limits_{j = 1}^m \theta_j w_j(T)\right| \leq \min_{x \in \R} \Big( \max \lr*{ \left| B_1 - xB \right|, \left| B_2 - xB \right| }\Big),
\]
where 
\[
\theta_j = t_j - \argmin\limits_{x \in \R} \Big( \max \lr*{ \left| B_1 - xB \right|, \left| B_2 - xB \right| }\Big).
\]
We have $B_1 \geq B_2$, since $k_j'$ is increasing. If $B_1 = B_2$, then the minimum is achieved at $x = B_1/B = (B_1 + B_2)/(2B)$. If $ B_1 > B_2 $, then the minimum is achieved when $\ B_2 - xB \leq 0 \leq B_1 - xB$. Indeed, if $B_1 - xB$ and $B_2 - xB$ are both positive (respectively, both negative), then a sufficiently small shift of $x$ to the right (respectively, to the left) will decrease $\max \lr*{ \left| B_1 - xB \right|, \left| B_2 - xB \right| }$. Then
\[
\min_{x \in \R}  \Big(\max \lr*{ \left| B_1 - xB \right|, \left| B_2 - xB \right| } \Big) = \min_{x \in \R}  \Big(\max \lr*{  B_1 - xB ,  xB - B_2  } \Big).
\]
This minimum is attained when $B_1 - xB = xB - B_2$, since it is minimum of two functions, one increasing and one decreasing. Now $x = (B_1 + B_2)/(2B)$ and
\begin{gather*}
    \left|\sum\limits_{j = 1}^m \theta_j w_j(T)\right| \leq \frac{B_1 - B_2}{2} = \frac{1}{2} \sum_{j = 1}^m t_j(k'_j - k'_{m - j + 1}) = \frac{1}{2} \sum_{j = 1}^{\left\lfloor \frac{m}{2} \right\rfloor} (t_j - t_{m - j + 1})(k'_j - k'_{m - j + 1}) \leq \\ \leq \frac{\diam Z}{2} \sum_{j = 1}^{\left\lfloor \frac{m}{2} \right\rfloor} (k'_{m - j + 1} - k'_j) = \frac{\diam Z}{2} \sum_{j = 1}^{\left\lfloor \frac{m}{2} \right\rfloor} (k_{m - j + 1} - m + j - 1 + 1 - k_j + j - 1) = \\ = \frac{\diam Z}{2} \sum_{j = 1}^{\left\lfloor \frac{m}{2} \right\rfloor} (k_{m - j + 1} - k_j - m + 2j - 1) = \frac{\diam Z}{2} \lr*{ \sum_{j = 1}^{\left\lfloor \frac{m}{2} \right\rfloor} (k_{m - j + 1} - k_j) -  \left\lfloor \frac{m^2}{4} \right\rfloor } = \\
    = \frac{\diam Z}{4} \lr*{ 2 \sum_{j = 1}^{\left\lfloor \frac{m}{2} \right\rfloor} (k_{m - j + 1} - k_j) -  \left\lfloor \frac{m^2}{2} \right\rfloor } < \frac{1}{4}.
\end{gather*}

\section{Examples}
\begin{example}

Let $A=\{0,1,3\}$ and let $w\in L^2(0,1)$ be a weight whose zero multiset consists of three distinct points $t_1,t_2,t_3\in\mathbb T$, each of multiplicity one. Write $z_j=e^{2\pi it_j}$. The partition associated with $A$ in Theorem 3 is $A'=(1,1,2)$, and $s_{(1,1,2)}(z_1,z_2,z_3)
=z_1z_2z_3(z_1+z_2+z_3)$. Consequently,
\[
\{w(t)e^{2\pi int}\}_{n\in\mathbb Z\setminus A}
\text{ is complete and minimal}
\quad\Longleftrightarrow\quad
z_1+z_2+z_3\ne0.
\]
Therefore the weighted system is exact precisely when the zero locations are not equally spaced around $\mathbb T$. This also gives an exact diameter criterion:
\[
\{w(t)e^{2\pi int}\}_{n\in\mathbb Z\setminus\{0,1,3\}}
\text{ is exact}
\quad\Longleftrightarrow\quad
\operatorname{diam}Z<\frac23.
\]
Indeed, let $g_1,g_2,g_3$ be the successive circular gaps between the zero locations. Then $\diam Z = 1-\max(g_1,g_2,g_3) \leq 2/3$, with equality precisely when $g_1=g_2=g_3=1/3$. Corollary \ref{cor:schur}, however, guarantees exactness for this defect only when $\diam Z < 1/2$. The Theorem \ref{th:fin_schur} therefore establishes exactness for this defect set $A$ beyond the range covered by the uniform diameter estimate.
\end{example}

\begin{example}

Let $d \ge 2$ and $A=\{0,d\}$. Suppose that the zero multiset consists of two distinct simple nodes $z_1,z_2$. The interpolation determinant is $z_2^d-z_1^d$. Writing $\delta = \diam Z\in(0,1/2]$, Theorem \ref{th:fin_schur} therefore gives that exactness equivalent to $d\delta\notin\mathbb Z$. The sharp uniform diameter guarantee for this defect is $\delta<1/d$. Every configuration satisfying this inequality is exact, whereas the nodes $0,1/d$ provide a singular configuration at equality. 

Corollary \ref{cor:schur} gives the more restrictive condition $\delta<\frac1{2(d-1)}$.
The two thresholds agree for $d=2$, but differ for every $d \geq 3$. Thus Corollary \ref{cor:schur} supplies an explicit sufficient condition valid for arbitrary defect sets and zero multiplicities. Its threshold need not be optimal for a fixed defect set. The boundary example with $A=\{0,2\}$ shows that the strict inequality cannot universally be replaced by a non-strict one.
    
\end{example}

\end{document}